\documentclass[11pt, reqno]{amsart}
\usepackage{color}
\usepackage{amsmath}
\usepackage{amssymb}
\usepackage{amsfonts}
\usepackage{mathrsfs}
\usepackage{amsbsy}
\usepackage{relsize}
\usepackage[colorlinks, linkcolor=red, citecolor=blue, urlcolor=blue, pagebackref, hypertexnames=false]{hyperref}
\usepackage[hyperpageref]{backref}
\usepackage[lmargin=3cm,rmargin=3cm,tmargin=3cm,bmargin=3cm]{geometry}
\def\R{\mathbb{R}}

\def\eps{\varepsilon}
\def\vv{\mathbf{V}}
\def\mm{\mathbf{M}}
\def\kk{\mathbf{K}}
\def\Ibold{\mathbf{I}}
\def\Jbold {\mathbf{J}}
\numberwithin{equation}{section}

\newtheorem{theo}{Theorem}[section]
\newtheorem{prop}[theo]{Proposition}
\newtheorem{lem}[theo]{Lemma}
\newtheorem{rem}[theo]{Remark}

\newtheorem{defi}[theo]{Definition}
\newtheorem{cor}[theo]{Corollary}
\numberwithin{equation}{section}

\newcommand{\beq}{\begin{equation}}
\newcommand{\eeq}{\end{equation}}

\newcommand{\bean}{\begin{eqnarray}}
\newcommand{\eean}{\end{eqnarray}}

\newcommand{\bea}{\begin{eqnarray*}}
\newcommand{\eea}{\end{eqnarray*}}

\newcommand{\bd}{\begin{description}}
\newcommand{\ed}{\end{description}}

\newcommand{\bc}{\begin{center}}
\newcommand{\ec}{\end{center}}

\newcommand{\ben}{\begin{enumerate}}
\newcommand{\een}{\end{enumerate}}

\newcommand{\bit}{\begin{itemize}}
\newcommand{\eit}{\end{itemize}}

\author{ Abdelwahab Bensouilah}

\address{Laboratoire Paul Painlev\'e (U.M.R. CNRS 8524), U.F.R. de Math\'ematiques, Universit\'e Lille 1, 59655 Villeneuve d'Ascq Cedex, France
}
\email{\sl ai.bensouilah@math.univ-lille1.fr}

\author{Van Duong Dinh}
\address{Institut de Mathematiques de Toulouse UMR5219, Universit\'e de Toulouse CNRS, 31062 Toulouse Cedex 9, France}
\email{\sl dinhvan.duong@math.univ-toulouse.fr}

\author{Mohamed Majdoub}
\address{Department of Mathematics, College of science, Imam Abdulrahman Bin Faisal University, P. O. Box 1982, Dammam, Saudi Arabia}
\email{\sl mmajdoub@iau.edu.sa}

\title[Scattering in the weighted $L^2$-space for exponential NLS] {Scattering in the weighted $L^2$-space for a 2D nonlinear Schr\"odinger equation with inhomogeneous exponential nonlinearity}

\date{\today}

\begin{document}
	\begin{abstract}
		We investigate the defocusing inhomogeneous  nonlinear Schr\"odinger equation
		$$
		i \partial_tu + \Delta u = |x|^{-b} \left({\rm e}^{\alpha|u|^2} - 1- \alpha |u|^2 \right) u, \quad u(0)=u_0, \quad x \in \R^2,
		$$
		with $0<b<1$ and $\alpha=2\pi(2-b)$. First we show the decay of global solutions by assuming that the initial data $u_0$ belongs to the weighted space $\Sigma(\R^2)=\{\,u\in H^1(\R^2) \ : \ |x|u\in L^2(\R^2)\,\}$. Then we combine the local theory with the decay estimate to obtain scattering in $\Sigma$ when the Hamiltonian is below the value $\frac{2}{(1+b)(2-b)}$.
	\end{abstract}
	
	\subjclass[2000]{35-xx, 35L70, 35Q55, 35B40, 35B33, 37K05, 37L50}
	
	\keywords{Inhomogeneous nonlinear Schr\"odinger equation; Decay solutions; Virial identity; Scattering; Weighted $L^2$-space; Exponential nonlinearity, Singular Moser-Trudinger inequality.}

	\maketitle
	
	
	\section{Introduction and main result}
	\label{S1}
	This paper is concerned with the scattering theory for the following initial value problem
	\begin{align}
		\left\{
		\begin{array}{rcl}
			i\partial_t u + \Delta u &=& |x|^{-b} \left({\rm e}^{\alpha|u|^2} - 1- \alpha |u|^2 \right) u,\\
			u(0) &=& u_0,
		\end{array}
		\right.
		\label{NLSexp}
	\end{align}
	where $u=u(t,x)$ is a complex-valued function in space-time $\R\times\R^2$, $0<b<1$ and $\alpha=2\pi(2-b)$.
	
	The classical nonlinear Schr\"odinger equation ($b=0$)  with pure power or exponential nonlinearities arises in various physical contexts, as for example the self trapped beams in plasma, the propagation of a laser beam, water waves at the free surface of an ideal fluid and plasma waves (see \cite{LLT}).
	
	From the mathematical point of view, the classical NLS equation, i.e.,  problem \eqref{NLSexp} with $b=0$, has attracted considerable attention in the mathematical community and the well-posedness theory as well as the scattering has been extensively studied, see for instance \cite{Azz, BIP, Caz, CIMM, NL, NO3}. We refer the reader to \cite{Ca, Tao} and references therein for more properties and
	information on nonlinear Schr\"odinger equations.
	
	In particular, in \cite{CIMM} a notion of criticality was proposed and the authors established in both subcritical and critical regimes the existence of global solutions in the functional space $C(\R,H^1(\R^2))\cap L^4_{loc}(\R, W^{1,4}(\R^2))$. Later on in \cite{NL}, the scattering in the energy space was obtained in the subcritical case. Note that the critical case was investigated in \cite{BIP} where the scattering is proved in the radial framework.
	
	The situation in the case $b>0$ is less understood. Recently, in \cite{BDM} the authors established the global well-posedness in the energy space for $0<b<1$. A natural question to ask then is the long time behavior of global solutions, that is the scattering. This means that every global solution of \eqref{NLSexp} approaches solutions to the associated free equation
	\begin{equation}
	\label{SL}
	i\partial_t v + \Delta v=0,
	\end{equation}
	in the energy space $H^1$ as $t\to\pm\infty$. The main difficulty is how to obtain the interaction Morawetz inequality? Recall that the interaction Morawetz inequality is nothing but the convolution of the classical one with the mass density. This in particular leads to a priori global bound of the solution in $L^4_t (L^8_x)$ which is the main tool for the scattering in the energy space (see for instance \cite{BIP,NL, PV}). Note that the interaction Morawetz inequalities were first established for the NLS with power-type nonlinearity, and the proof depends heavily on the form of nonlinearity. Of course the proof can be easily adapted to more general homogeneous nonlinearities. More precisely, for linear combination of powers it suffices that all the powers are quadratic or higher with positive coefficients. The problem with singular weight (or for non-homogeneous nonlinearity) is much more difficult and should be investigated separately. For instance, it was noticed in \cite{Dinh-energy} that the interaction Morawetz inequality for the NLS with singular nonlinearity $N(x,u) =|x|^{-b} |u|^\alpha u$ may not hold due to the lack of momentum conservation law.
	
	This is why we restrict ourselves to initial data belonging to the weighted $L^2$-space $\Sigma:=H^1 \cap L^2(|x|^2 dx)$. Note that the scattering in $\Sigma$ for the NLS with $N(x,u) = |x|^{-b} |u|^\alpha u$ was considered by the second author in \cite{Dinh}.
	
	The scattering in the energy space will be investigated in a forthcoming paper, and we believe that some ideas developed in \cite{BIP} will be helpful.
	
	\begin{rem}
		We stress that the two-dimensional nonlinear Klein-Gordon equation with pure exponential nonlinearity was studied in \cite{CPAM, Duke, APDE}, and a similar trichotomy based on the energy was defined. Recently, M. Struwe \cite{Struwe1, Struwe2} was able to construct global smooth solution for smooth initial data and prove the scattering \cite{Struwe0}.
	\end{rem}
	
	Before stating our main result, let us recall that solutions of \eqref{NLSexp} satisfy the conservation of mass
	and Hamiltonian
	\begin{equation}
	\label{mass} {\mathcal M}(u(t)):=\|u(t)\|_{L^2},
	\end{equation}
	\begin{equation}
	\label{hamil} {\mathcal H}(u(t)):=\int|\nabla u(t,x)|^2 +\frac{1}{\alpha}\int\left({\rm e}^{\alpha|u(t,x)|^2}-1-\alpha|u(t,x)|^2-\frac{\alpha^2}{2}|u(t,x)|^4 \right) \frac{dx}{|x|^b}.
	\end{equation}
	Our main result is the following.
	\begin{theo}
		\label{Main-NLS}
		Let $u_0 \in \Sigma$ be such that $\mathcal{H}(u_0) <\frac{2}{(1+b)(2-b)}$. Then the corresponding global solution $u$ of \eqref{NLSexp} satisfies $u \in L^4(\R, \mathcal{C}^{1/2})$ and there exist $u_0^\pm \in \Sigma$ such that
		\[
		\lim_{t\rightarrow \pm \infty} \|e^{-it\Delta} u(t) - u_0^\pm\|_{\Sigma} =0. 
		\]
	\end{theo}
	Let us make some comments. First, we see that $\frac{2}{(1+b)(2-b)} \rightarrow 1$ as $b \rightarrow 0$. Thus our result extends the one in \cite{NL} for initial data in $\Sigma$. Second, the condition $\mathcal{H}(u)<\frac{2}{(1+b)(2-b)}$ illustrates the interaction between the wave function $u$ and the potential $|x|^{-b}$. More precisely, a sufficient condition for scattering is when the energy of the wave is less than a fixed amount depending on the sole parameter $b$ that characterizes the weight function involved in the Hamiltonian of \eqref{NLSexp}. Finally, a natural question that one could raise is the following: is the value $\frac{2}{(1+b)(2-b)}$
	critical for scattering, in the sense that if the energy of the wave exceeds the latter quantity, would one get scattering?
	\begin{rem}
		For all $0<b<1$, \quad $\frac{8}{9}\leq \frac{2}{(1+b)(2-b)}<1$.
	\end{rem}
	The proof of Theorem $\ref{Main-NLS}$ follows a standard strategy for the classical NLS equation. We first derive a decaying property for global solutions by using the pseudo-conformation law. We then show two types of global bounds for the solution $u$ and its weighted variant $(x+2it\nabla) u$. More precisely, we will show that
    \begin{align} \label{glo-bou}
    \|u\|_{S^1(\R)}<\infty, \quad \|(x+2it\nabla) u\|_{S^0(\R)}<\infty,
    \end{align}
    where 
    \[
    \|u\|_{S^1(\R)} := \|u\|_{L^\infty(\R, H^1)} + \|u\|_{L^4(\R, W^{1,4})}, \quad \|u\|_{S^0(\R)}:=\|u\|_{L^\infty(\R, L^2)} + \|u\|_{L^4(\R, L^4)}.
    \]
    The proof of these global bounds relies on the decaying property, the singular Moser-Trudinger inequality and the Log estimate. The main difficulty comes from the singular weight $|x|^{-b}$ which does not belong to any Lebesgue space. To overcome this problem, we will take the advantage of Lorentz spaces. Note that $|x|^{-b} \in L^{\frac{2}{b},\infty}(\R^2)$, where $L^{p,\infty}$ is the Lorentz space. Once these global bounds are established, the scattering in weighted $L^2$ space $\Sigma$ follows easily. 

    This paper is organized as follows. In Section \ref{S2}, we recall some useful tools needed in our problem. The pseudo-conformal law is derived in Section \ref{S3}. The decaying property of global solutions in Lebesgue spaces is showed in Section \ref{S4}. Sections \ref{S5} and \ref{S6} are devoted to the proofs of global bounds \eqref{glo-bou}. We shall give the proof of our main result in Theorem $\ref{Main-NLS}$ in Section \ref{S7}. 

	\section{Useful Tools}
	\label{S2}
	In this section, we collect some known and useful tools.
	\begin{prop}[\sf Moser-Trudinger inequality \cite{AT}]\quad\\
		\label{prop-MT}
		Let $\alpha\in [0,4\pi)$. A constant $c_\alpha$ exists
		such that
		\begin{equation}
		\label{MT1} \|\exp(\alpha |u|^2)-1\|_{L^1(\R^2)}\leq c_\alpha
		\|u\|_{L^2(\R^2)}^2,
		\end{equation}
		for all $u$ in $H^1(\R^2)$ such that $\|\nabla
		u\|_{L^2(\R^2)}\leq1$. Moreover, if $\alpha\geq 4\pi$, then
		(\ref{MT1}) is false.
	\end{prop}
	\begin{rem}
		\label{rem} We point out that $\alpha=4\pi$ becomes admissible in
		(\ref{MT1}) if we require $\|u\|_{H^1(\R^2)}\leq1$ rather than
		$\|\nabla u\|_{L^2(\R^2)}\leq1$. Precisely, we have
		\beq
		\label{MT2}
		\sup_{\|u\|_{H^1}\leq
			1}\;\;\|\exp(4\pi |u|^2)-1\|_{L^1(\R^2)}<\infty,
		\eeq
		and this is
		false for $\alpha>4\pi$. See \cite{Ru} for more details.
	\end{rem}
	\begin{theo}\cite{Souza1}\label{theo-MT}
		Let $0< b <2$ and $0<\alpha<2\pi(2-b)$. Then, there exists a positive constant $C=C(b,\alpha)$ such that
		\begin{equation}
		\label{WMT1}
		\int_{\R^2}\frac{{\rm e}^{\alpha| u(x)|^2}-1}{| x|^b}dx\leqslant C\int_{\R^2}\frac{| u(x)|^2}{| x|^b}dx,
		\end{equation}
		for all $u\in H^1(\R^2)$ with $\| \nabla u\|_{L^2(\R^2)}\leqslant 1$.
	\end{theo}
	We point out that $\alpha=2\pi(2-b)$ becomes admissible in
	(\ref{WMT1}) if we require $\|u\|_{H^1(\R^2)}\leq1$ instead of
	$\|\nabla u\|_{L^2(\R^2)}\leq1$. More precisely, we have
	\begin{theo}\cite{Souza2}\label{MTT}
		Let $0< b <2$. We have
		\begin{equation}
		\label{WMT2}
		\sup_{\| u\|_{H^1(\R^2)}\leqslant 1}\int_{\R^2}\frac{{\rm e}^{\alpha| u(x)|^2}-1}{| x|^b}dx<\infty\quad\mbox{if and only if}\quad \alpha\leq 2\pi(2-b).
		\end{equation}
	\end{theo}
	The following lemma will be very useful.
	\begin{lem}
		\label{hardy}
		Let  $0< b <2$ and $\gamma \geq 2$. Then, there exists a positive constant $C=C(b,\gamma)>0$ such that
		\begin{equation}
		\label{Hardy}
		\int_{\R^2}\frac{|u(x)|^{\gamma}}{|x|^b} dx\leq\,C \|u\|_{H^1(\R^2)}^{\gamma},
		\end{equation}
		for all $u\in H^1(\R^2)$.
	\end{lem}
	\begin{proof}
		Note that
		\begin{equation}
		\label{integ}
		\||x|^{-b}\|_{L^r(B)} <\infty\;\;\mbox{if}\;\; b<\frac{2}{r},\quad \||x|^{-b}\|_{L^r(B^c)} <\infty\;\; \mbox{if}\;\; b>\frac{2}{r},
		\end{equation}
		where $B=B(0,1)$ is the unit ball in $\R^2$ and $B^c=\R^2\backslash B$. Write
		\[
		\int_{\R^2} |x|^{-b} |u(x)|^\gamma dx = \int_{B} |x|^{-b} |u(x)|^\gamma dx + \int_{B^c} |x|^{-b} |u(x)|^\gamma dx.
		\]
		We have from the Sobolev embedding $H^1(\R^2) \subset L^q(\R^2)$ for any $q \in[2,\infty)$ that
		\[
		\int_{B^c} |x|^{-b} |u(x)|^\gamma dx \leq \|u\|^\gamma_{L^{\gamma}(\R^2)} \lesssim \|u\|^\gamma_{H^1(\R^2)}.
		\]
		The first term is estimated as follows. Since $0<b<2$, there exists $\eps>0$ small such that $b<\frac{2}{1+\eps}$. We apply \eqref{integ} with $r=1+\eps$ and get
		\[
		\int_{B} |x|^{-b} |u(x)|^\gamma dx \leq \||x|^{-b}\|_{L^{1+\eps}(B)} \||u|^\gamma\|_{L^{\frac{1+\eps}{\eps}}(\R^2)} \lesssim \|u\|^\gamma_{L^{\frac{(1+\eps)\gamma}{\eps}}(\R^2)} \lesssim \|u\|^\gamma_{H^1(\R^2)}.
		\]
		Combining the two terms, we prove the desired estimate.
	\end{proof}
	\begin{rem}
		The inequality \eqref{Hardy} fails for $b\geqslant2$. Indeed, let $u\in \mathcal{D}(\R^2)$ (the space of smooth compactly supported functions) be a radial function such that
		$u(x)\equiv1$ for $| x| \leqslant 1$. Then, $u\in H^1(\R^2)$ and
		$$\int_{\R^2}\frac{| u(x)|^\gamma}{| x|^b} dx\geqslant 2\pi \int_0^1\frac{rdr}{r^b}=+\infty.$$
	\end{rem}
	We also recall the so-called Gagliardo-Nirenberg inequalities and Sobolev embedding.
	\begin{prop}[\sf Gagliardo-Nirenberg inequalities \cite{Gag, Nir}] \quad\\
		\label{GN}
		We have
		\begin{equation}
		\label{GaNi}
		\|u\|_{L^{m+1}}\lesssim\, \|u\|_{L^{q+1}}^{1-\theta}\, \|\nabla u\|_{L^{p}}^\theta,
		\end{equation}
		where
		$$
		\theta=\frac{pN(m-q)}{(m+1)[N(p-q-1)+p(q+1)]}, \;\;\; 0\leq q <\sigma-1,\;\;\; q<m<\sigma,
		$$
		\begin{eqnarray*}
			\sigma&=&\; \left\{
			\begin{array}{cllll} \frac{(p-1)N+p}{N-p}\quad&\mbox{if}&\quad
				p<N\\ \infty \quad
				&\mbox{if}&\quad  p\geq N
			\end{array}
			\right.
		\end{eqnarray*}
	\end{prop}
	In particular, for $N=2$, we obtain
	\begin{equation}
	\label{GaNi2D}
	\|u\|_{L^q}\lesssim\, \|u\|_{L^2}^{2/q}\,\, \|\nabla u\|_{L^2}^{1-2/q},\quad 2\leq\,q<\infty.
	\end{equation}
	\begin{prop}[\sf Sobolev embeddings] \quad\\
		\label{SobEmb}
		We have
		\beq
		\label{Sob}
		W^{s,p}(\R^N)\hookrightarrow L^q(\R^N),\quad 1\leq p<\infty,\quad 0\leq s<\frac{N}{p},\quad \frac{1}{p}-\frac{s}{N}\leq \frac{1}{q}\leq \frac{1}{p}.
		\eeq
		\beq
		\label{Sobb}
		W^{1,p}(\R^N)\hookrightarrow \mathcal{C}^{1-{N\over p}}(\R^N),\quad p>N.
		\eeq
	\end{prop}
	The following estimate is an $L^\infty$ logarithmic inequality
	which enables us to establish the link between $ \|{\rm e}^{4\pi
		|u|^2}-1\|_{L^1_T(L^2(\R^2))} $ and dispersion properties of
	solutions of the linear Schr\"odinger equation.
	\begin{prop}[\sf Log estimate \cite{IMM}]\quad\\
		\label{LogEst}
		Let $0<\beta<1$. For any $\lambda>\frac{1}{2\pi\beta}$ and
		any $0<\mu\leq1$, a constant $C_{\lambda}>0$ exists such
		that,
		for any function $u\in H^1(\R^2)\cap{\mathcal C}^\beta(\R^2)$, we have
		\begin{equation}
		\label{Log}
		\|u\|^2_{L^\infty}\leq
		\lambda\|u\|_\mu^2 \log\left(C_{\lambda} +
		\frac{8^\beta\mu^{-\beta}\|u\|_{{\mathcal C}^{\beta}}}{\|u\|_\mu}\right),
		\end{equation}
		where 
		\begin{equation}
		\label{Hmu}
		\|u\|_\mu^2:=\|\nabla u\|_{L^2}^2+\mu^2\|u\|_{L^2}^2.
		\end{equation}
	\end{prop}
	Recall that ${\mathcal C}^{\beta}(\R^2)$ denotes the space of
	$\beta$-H\"older continuous functions endowed with the norm $$
	\|u\|_{{\mathcal
			C}^{\beta}(\R^2)}:=\|u\|_{L^\infty(\R^2)}+\sup_{x\neq
		y}\frac{|u(x)-u(y)|}{|x-y|^{\beta}}. $$
	We refer to  \cite{IMM} for the proof of this proposition and more
	details. We just point out that the condition  $ \lambda>
	\frac{1}{2\pi\beta}$ in (\ref{Log}) is optimal.
	
	We also recall the so-called Strichartz estimates. We say that $(q,r)$ is an $L^2$-admissible pair if
	\begin{equation}
	\label{admiss}
	0\leq\frac{2}{q}=1-\frac{2}{r}<1.
	\end{equation}
	In particular, note that $(\frac{2}{1-2\sigma},\frac{1}{\sigma})$ is an admissible pair for any $0<\sigma<1/2$ and $${W}^{1,\frac{1}{\sigma}}(\R^2)\hookrightarrow {\mathcal C}^{1-2\sigma}(\R^2).$$
	\begin{prop}[\sf Strichartz estimates \cite{Ca}]\quad\\
		\label{Strich} Let $I\subset\R$ be a time interval and let $t_0\in I$. Then, for any admissible pairs $(q,r)$ and $(\tilde{q}, \tilde{r})$, we have
		\begin{equation}
		\label{stri} \|v\|_{L^q(I,{W}^{1,r}(\R^2))}\lesssim\,\|v(t_0)\|_{H^1(\R^2)}+\| i \partial_t v+\Delta v\|_{L^{\tilde{q}'}(I,W^{1,\tilde{r}'}(\R^2))}.
		\end{equation}
	\end{prop}

	The following continuity argument (or bootstrap argument) will be useful for our purpose.
	\begin{theo}[\sf Continuity argument] \quad\\
		\label{cont}
		Let $X : [0,T]\to\R$ be a nonnegative continuous function, such that, for
		every $0\leqslant t \leqslant T $,
		$$
		X(t) \leqslant a+ b X(t)^{\theta} \, ,
		$$
		where $a,b>0$ and $\theta>1$ are constants such that
		$$
		a<\left(1-\frac{1}{\theta}\right) \frac{1}{(\theta b)^{1/(\theta-1)}} \quad \mbox{and} \quad X(0)\leqslant \frac{1}{(\theta b)^{1/(\theta-1)}}.
		$$
		Then, for every $0\leqslant t \leqslant T $, we have
		$$
		X(t) \leqslant \frac{\theta}{\theta-1}a.
		$$
	\end{theo}
	\begin{proof}
    We sketch the proof for reader's convenience. The function $f: x \longmapsto b x^{\theta}-x+a$ is decreasing on $[0,(\theta b)^{1/(1-\theta)} ]$ and increasing on $[(\theta b)^{1/(1-\theta)},\infty)$. The assumptions on $a$ and $X(0)$ imply that $f((\theta b)^{1/(1-\theta)})<0$. As $f(X(t))\geqslant 0, f(0) >0$ and $X(0)\leqslant \frac{1}{(\theta b)^{1/(\theta-1)}}$, we deduce the desired result.
    \end{proof}

	\section{Pseudo-conformal law}
	\label{S3}
	In this section, we show a decaying property of global solutions to \eqref{NLSexp}. Note that the conservation laws of mass and Hamiltonian give the boundedness of the $L^2$ and the $H^1$ norms but are insufficient to provide a decay estimate in (more general) Lebesgue spaces. To obtain such a decay we will take advantage of the pseudo-conformal law.
	
	More precisely, we define the following quantities
	\beq
	\label{V}
	\vv(t):=\int\,|x|^2 |u(t,x)|^2\,dx,
	\eeq
	
	\beq
	\label{M}
	\mm(t):=2\int\,{\mathcal{I}} \left(\bar{u}(t,x) x\cdot\nabla u (t,x)\right)\,dx,
	\eeq
	
	\beq
	\label{K}
	\kk(t):=\|(x+2i t \nabla)u(t)\|_{L^2}^2+\frac{4t^2}{\alpha}\int\left({\rm e}^{\alpha|u(t,x)|^2}-1-\alpha|u(t,x)|^2-\frac{\alpha^2}{2}|u(t,x)|^4 \right) \frac{dx}{|x|^b},
	\eeq
	\begin{align} \label{def-G}
			\begin{aligned}
				G(t) &:=\frac{4(2-b)}{\alpha} \int \left( {\rm e}^{\alpha|u(t,x)|^2} - 1 - \alpha |u(t,x)|^2-\frac{\alpha}{2} |u(t,x)|^4\right) \frac{dx}{|x|^b}  \\
				&\mathrel{\phantom{=}} -\frac{8}{\alpha} \int \left( {\rm e}^{\alpha|u(t,x)|^2} (\alpha|u(t,x)|^2-1) +1 - \frac{\alpha^2}{2} |u(t,x)|^4\right) \frac{dx}{|x|^b} \\
				&=: \int g(|u(t,x)|^2) \frac{dx}{|x|^b},
			\end{aligned}
		\end{align}
		where
		\begin{align}
			\label{def-g}
			g(\tau) = \frac{4(2-b)}{\alpha} \left( {\rm e}^{\alpha \tau} -1 -\alpha \tau -\frac{\alpha^2}{2} \tau^2\right) -\frac{8}{\alpha} \left({\rm e}^{\alpha\tau} (\alpha \tau-1) +1 - \frac{\alpha^2}{2} \tau^2 \right).
		\end{align}
		 
	\begin{prop}
		\label{VMKG}
		Let $u_0 \in \Sigma$ and $u$ the corresponding global solution to \eqref{NLSexp}. Then
		\begin{eqnarray}
		\frac{d\vv(t)}{dt}&=&2\mm(t),\label{dv}\\
		\frac{d^2\vv(t)}{dt^2}&=&8{\mathcal H}(u_0)-G(t), \label{ddv}\\
		\frac{d\kk(t)}{dt}&=&tG(t),\label{dk}\\
		G(t)&\leq& 0,\quad \forall\; t\in\R. \label{Gl0}
		\end{eqnarray}
	\end{prop}
	
	\begin{proof}
		A straightforward computation gives \eqref{dv}.  Let $N(x,u):=|x|^{-b} \left({\rm e}^{\alpha|u|^2} - 1- \alpha |u|^2 \right) u$. Following \cite{TVZ} for instance, we find that
		$$
		\frac{d^2\vv(t)}{dt^2}=8\int\,|\nabla u|^2\,dx+4\int\, x\cdot\left\{N(x,u), u\right\}_p\,dx,
		$$
		where $\left\{f, g\right\}_p={\mathcal R}\left(f\nabla\bar{g}-g\nabla\bar{f}\right)$ is the momentum bracket.

		Now compute the momentum bracket $\left\{N(x,u), u\right\}_p$.  Expand $N(x,u)$ in a formal series
		$$
		N(x,u)=|x|^{-b}\sum_{k=2}^{\infty}\, \frac{\alpha^k}{k!}\,|u|^{2k}u.
		$$
		Using the fact 
			\[
			\{|x|^{-b} |u|^\beta u, u\}_p = -\frac{\beta}{\beta+2} \nabla (|x|^{-b} |u|^{\beta+2}) - \frac{2}{\beta+2} \nabla(|x|^{-b}) |u|^{\beta+2},
			\]
		one gets
		\begin{align*}
				\{N(x,u),u\}_p &= \sum_{k=2}^\infty \frac{\alpha^k}{k!} \{|x|^{-b} |u|^{2k} u, u\}_p \\
				&= - \sum_{k=2}^\infty k\frac{\alpha^k}{(k+1)!} \nabla (|x|^{-b} |u|^{2k+2}) - \sum_{k=2}^\infty \frac{\alpha^k}{(k+1)!} \nabla(|x|^{-b}) |u|^{2k+2}.
			\end{align*}
		An integration by parts leads
		\begin{align*}
				\int x \cdot \{ N(x,u),u\}_p &= 2\int \left( \sum_{k=2}^\infty k \frac{\alpha^k}{(k+1)!} |u|^{2k+2} \right) \frac{dx}{|x|^b} + b \int \left( \sum_{k=2}^\infty \frac{\alpha^k}{(k+1)!} |u|^{2k+2} \right) \frac{dx}{|x|^b} \\
				&=\frac{2}{\alpha} \int \left( {\rm e}^{\alpha|u|^2} (\alpha|u|^2-1) +1 - \frac{\alpha^2}{2}|u|^4\right) \frac{dx}{|x|^b} \\
				&\mathrel{\phantom{=}}+ \frac{b}{\alpha} \int \left( {\rm e}^{\alpha|u|^2} -1 - \alpha|u|^2 - \frac{\alpha^2}{2} |u|^4 \right) \frac{dx}{|x|^b},
			\end{align*}
		where we have used
			\begin{align*}
				\sum_{k=2}^\infty k \frac{\alpha^k}{(k+1)!} |u|^{2k+2} &= \frac{1}{\alpha} \left({\rm e}^{\alpha|u|^2} (\alpha |u|^2-1) +1 - \frac{\alpha^2}{2} |u|^4 \right), \\
				\sum_{k=2}^\infty \frac{\alpha^k}{(k+1)!} |u|^{2k+2} &= \frac{1}{\alpha} \left( {\rm e}^{\alpha|u|^2} -1 - \alpha |u|^2 - \frac{\alpha^2}{2}|u|^4\right).
			\end{align*}
			Therefore,
			\begin{align*}
				\frac{d^2\vv(t)}{dt^2} &= 8 \|\nabla u(t)\|^2_{L^2} + \frac{8}{\alpha} \int \left({\rm e}^{\alpha|u|^2}(\alpha|u|^2-1) +1 - \frac{\alpha^2}{2} |u|^4\right) \frac{dx}{|x|^b} \\
				&\mathrel{\phantom{= 8 \|\nabla u(t)\|^2_{L^2} }} + \frac{4b}{\alpha} \int \left( {\rm e}^{\alpha|u|^2} - 1- \alpha|u|^2 - \frac{\alpha^2}{2} |u|^4\right) \frac{dx}{|x|^b}.
			\end{align*}
		Using the conservation law \eqref{hamil}, we conclude the proof of \eqref{ddv}. To prove \eqref{dk}, we first remark that
		$$
		\kk(t)=\vv(t)-t\frac{d\vv(t)}{dt}+4t^2{\mathcal H}(u_0).
		$$
		Hence
		$$
		\frac{d\kk(t)}{dt}=-t\frac{d^2\vv(t)}{dt^2}+8t{\mathcal H}(u_0),
		$$
		and the conclusion follows. Finally, for the sign of $G$, a simple computation shows that (for all $\tau\geq 0$)
		\begin{eqnarray*}
			g'(\tau)&=&-8(\alpha xe^{\alpha x}-e^{\alpha x}+1)-4 b (e^{\alpha x}-\alpha x-1)\leq 0.
		\end{eqnarray*}
		Since $g(0)=0$, we get \eqref{Gl0}.
	\end{proof}
	As a consequence of Proposition \ref{VMKG}, we have
	\begin{cor}
		\label{KK}
		Let $u_0 \in \Sigma$ and $u$ the corresponding global solution to \eqref{NLSexp}. Then
		\begin{multline*}
			\|(x+2it\nabla)u(t)\|_{L^2}^2+\frac{4t^2}{\alpha}\int \left({\rm e}^{\alpha|u(t,x)|^2}-1-\alpha|u(t,x)|^2-\frac{\alpha^2}{2}|u(t,x)|^4 \right) \frac{dx}{|x|^b} \\
			=\|x u_0\|_{L^2}^2+\int_{0}^{t} \tau\,G(\tau) \, d\tau.
		\end{multline*}
	\end{cor}
	
	
	\section{Decay estimate}
	\label{S4}
	\begin{theo}
		Let $u_0 \in \Sigma$ and $u$ the corresponding global solution to \eqref{NLSexp}. Then, for all $t\neq0$ and $2\leq\,q<\infty$,
		\[
		\| u(t)\|_{L^q} \leq\,C_q\|u_0\|_{\Sigma}\, |t|^{-(1-\frac{2}{q})},
		\]
		where $C_q>0$ is a constant depending only on $q$.
	\end{theo}
	\begin{proof}
		Set $v(t,x):={\rm e}^{-i\frac{|x|^2}{4t}}u(t,x)$. We see that
		$\|(x+2it\nabla)u(t)\|_{L^2}^2=4t^2 \|\nabla v(t)\|_{L^2}^2$. Hence, by Corollary \ref{KK},
		\[
		4 t^2 \mathcal{H} (v(t))=\|x u_0\|_{L^2}^2+\int_{0}^{t} \tau G(\tau) \, d\tau.
		\]
		Using \eqref{Gl0}, we get
		\[
		4 t^2 \|\nabla v(t)\|_{L^2}^2 \leq \|x u_0\|_{L^2}^2,
		\]
		or equivalently
		\[
		\|\nabla v(t)\|_{L^2} \lesssim |t|^{-1}.
		\]
		The conservation of mass, the fact that $|u|=|v|$ and the Gagliardo-Nirenberg inequality \eqref{GaNi2D}, yield, for all $2\leq\, q < \infty$,
		\[
		\| u(t)\|_{L^q}=\| v(t)\|_{L^q} \lesssim |t|^{-(1-\frac{2}{q})}.
		\]
		The proof is complete.
	\end{proof}
	
	A natural and useful consequence from the previous theorem is the following bound estimate.
	\begin{cor}
		\label{Bound1}
		Let $u_0 \in \Sigma$ and $u$ the corresponding global solution to \eqref{NLSexp}. Let $1\leq p<\infty,\;\; 2\leq q<\infty$ be such that
		\begin{equation}
		\label{pq}
		p\left(1-{2 \over q}\right)>1.
		\end{equation}
		Then, for all $T>0$, we have
		$$
		\|u\|_{L^p([T,\infty); L^q)}\lesssim \frac{T^{{1\over p}+{2\over q}-1}}{\left(p\left(1-{2\over q}\right)-1\right)^{1/p}}<\infty.
		$$
	\end{cor}
	For bounded time intervals, the local theory allows us to remove the assumption \eqref{pq} to obtain
	\begin{cor}
		\label{Bound2}
		Let $u_0 \in \Sigma$ and $u$ the corresponding global solution to \eqref{NLSexp}. Let $1\leq p<\infty,\;\; 2\leq q<\infty$ and $0<T<S<\infty$. Then
		$$
		\|u\|_{L^p([T,S]; L^q)}\leq C,
		$$
		where $C>0$ depends only on $p,\;\;q,\;\; T,\;\; S,\;\;\|u_0\|_{\Sigma}$.
	\end{cor}
	Another important consequence that will be used to obtain global bounds asserts that one can decompose any time interval $(T,\infty)$ with $T>0$ into a finite number of intervals on which  the $L^p_t(L^q_x)$ norm is sufficiently small for every $(p,q)$ satisfying \eqref{pq}. More precisely, we have
	\begin{cor}
		\label{Bound3}
		Let $u_0 \in \Sigma$ and $u$ the corresponding global solution to \eqref{NLSexp}. Let $1\leq p<\infty,\;\; 2\leq q<\infty$, $\varepsilon>0$ and $T>0$. Assume that the condition \eqref{pq} is fulfilled. Then there exists $L\geq 1$ not depending on $u$ and time intervals $I_1, I_2,\cdots, I_L$ such that $\displaystyle\bigcup_{\ell=1}^{L}\, I_{\ell}=[T,\infty)$ and
		\begin{equation}
		\label{unifB}
		\|u\|_{L^p(I_{\ell}; L^q)}\leq\varepsilon,\qquad \forall\;\;\; \ell=1,2,\cdots,L.
		\end{equation}
	\end{cor}
	\begin{proof}
		From Corollary \ref{Bound1}, one can choose $S>T$ sufficiently large (not depending on $u$) such that $\|u\|_{L^p([S,\infty); L^q)}\leq\varepsilon$. Define
		$$
		T_\ell=T+\ell\frac{S-T}{m},\qquad \ell=0,1,\cdots, m,
		$$
		where $m\geq 1$ to be chosen later. Using H\"older's inequality in time, we obtain that
		\begin{eqnarray*}
			\|u\|_{L^p(T_\ell, T_{\ell+1}]; L^q)}&\leq& \left(\frac{S-T}{m}\right)^{\frac{1}{2p}}\; \|u\|_{L^{2p}([T, S]; L^q)}\\
			&\lesssim& \left(\frac{S-T}{m}\right)^{\frac{1}{2p}}\\
			&\leq&\varepsilon,
		\end{eqnarray*}
		for $m\geq 1$ sufficiently large and for all $\ell=0,1,\cdots,m-1$. This finishes the proof of Corollary \ref{Bound3}.
	\end{proof}
	\section{Global bounds 1}
	\label{S5}
	In this section, we give the proof of the first global bound in \eqref{glo-bou}. For a time slab $I\subset\R$, we define $S^1(I)$ via
	$$
	\|u\|_{S^1(I)} = \|u\|_{L^\infty(I,H^1)} +  \|u\|_{L^4(I,W^{1,4})}.
	$$
	By the Strichartz estimates, we have
	\begin{equation}
	\label{SStr}
	\|u\|_{S^1(I)}\lesssim \|u(T)\|_{H^1}+\|i \partial_t u
	+ \Delta u\|_{L^{\frac{2}{1+2\delta}}(I, W^{1,{\frac{1}{1-\delta}}})},
	\end{equation}
	for any $0<\delta < 1/2$ and $T\in I$. Note that $\left( \frac{2}{1+2\sigma}, \frac{1}{1-\delta}\right)$ is the conjugate pair of the Schr\"odinger admissible pair $\left(\frac{2}{1-2\sigma}, \frac{1}{\sigma}\right)$.
	
	\begin{theo}
		Let $u_0 \in \Sigma$ be such that $\mathcal{H}(u_0)<\frac{2}{(1+b)(2-b)}$. Then the corresponding global solution $u$ to \eqref{NLSexp} satisfies $u \in S^1(\R)$.
	\end{theo}
	
	\begin{proof}
		It suffices to estimate the nonlinear term in some dual Strichartz
		norm as in \eqref{SStr}. We have
		\[
		|\nabla N(x,u)| \lesssim |x|^{-b} |\nabla u| |u|^2 \left({\rm e}^{\alpha|u|^2} - 1\right)+ |x|^{-b-1} |u| \left({\rm e}^{\alpha|u|^2} - 1- \alpha |u|^2 \right):= \mathcal{A}+\mathcal{B}.
		\]
		Let $0<\delta<\frac{1}{2}$ to be chosen adequately, and let $I$ be  a time slab. Let us first estimate the norm $\|\mathcal{A}\|_{L^{\frac{2}{1+2 \delta}}(I,L^{\frac{1}{1-\delta}})}$. By H\"older's inequality,
		\[
		\|\mathcal{A}\|_{L^{\frac{1}{1-\delta}}} \lesssim \|\nabla u\|_{L^{\frac{2}{1-\delta}}} \|u\|_{L^{\frac{4}{\delta}}}^2 \bigg\|\frac{{\rm e}^{\alpha|u|^2}-1}{|x|^{b}} \bigg\|_{L^{\frac{2}{1-2\delta}}}.
		\]
		The term $\bigg\|\frac{{\rm e}^{\alpha|u|^2}-1}{|x|^{b}} \bigg\|_{L^{\frac{2}{1-2\delta}}}$ can be estimated using Lorentz spaces. Indeed, by \eqref{Lor}, we get
		\begin{align*}
			\bigg\|\frac{{\rm e}^{\alpha|u|^2}-1}{|x|^{b}} \bigg\|_{L^{\frac{2}{1-2\delta}}} &\lesssim \|{\rm e}^{\alpha|u|^2}-1\|_{L^{1}}^{1-\theta} \|{\rm e}^{\alpha|u|^2}-1\|_{L^{\infty}}^{\theta} \||x|^{-b}\|_{L^{\frac{2}{b},\infty}}\\
			&\lesssim \|{\rm e}^{\alpha|u|^2}-1\|_{L^{\infty}}^{\theta} ,
		\end{align*}
		where $\theta:=\delta+\frac{1+b}{2}$. Note that we can choose $0<\delta<\frac{1-b}{2}$ so that $\theta \in (0,1)$. Here we have used the Moser-Trudinger inequality \eqref{MT1} to obtain that $\|{\rm e}^{\alpha|u|^2}-1\|_{L^{1}}\lesssim 1$ since $\|\nabla u\|^2_{L^2}  < \mathcal{H}(u_0) <\frac{2}{(1+b)(2-b)}<1$.
		Hence
		\begin{align*}
			\|\mathcal{A}\|_{L^{\frac{2}{1+2 \delta}}(I,L^{\frac{1}{1-\delta}})} &\lesssim \left\|\|\nabla u\|_{L^{\frac{2}{1-\delta}}} \|u\|_{L^{\frac{4}{\delta}}}^2 \|{\rm e}^{\alpha|u|^2}-1\|_{L^{\infty}}^{\theta} \right\|_{L^{\frac{2}{1+2 \delta}}(\Ibold)} \\
			&\mathrel{\phantom{\lesssim}} +\left\|\|\nabla u\|_{L^{\frac{2}{1-\delta}}} \|u\|_{L^{\frac{4}{\delta}}}^2 \|{\rm e}^{\alpha|u|^2}-1\|_{L^{\infty}}^{\theta}\right\|_{L^{\frac{2}{1+2 \delta}}(\Jbold)},
		\end{align*}
		where $\Ibold=\{t \in I / \, \|u(t)\|_{L^{\infty} } \leq 1\}$ and $\Jbold=\{t \in I / \, \|u(t)\|_{L^{\infty} } \geq 1\}$. The first term in the right hand side can be easily estimated as follows
		\begin{align}
		\label{A1}
		\left\|\|\nabla u\|_{L^{\frac{2}{1-\delta}}} \|u\|_{L^{\frac{4}{\delta}}}^2 \|{\rm e}^{\alpha|u|^2}-1\|_{L^{\infty}}^{\theta}\right\|_{L^{\frac{2}{1+2 \delta}}(\Ibold)} &\lesssim 
		\|\nabla u\|_{L^{\frac{2}{\delta}}(I,L^{\frac{2}{1-\delta}})}  \|u\|_{L^{\frac{4}{1+\delta}}(I,L^{\frac{4}{\delta}})}^{2} \nonumber\\
		&\lesssim  \|u\|_{S^1(I)}  \|u\|_{L^{\frac{4}{1+\delta}}(I,L^{\frac{4}{\delta}})}^{2},
		\end{align}
		where the following interpolation inequality is used
		\begin{align} \label{interpolation}
		\|\nabla u\|_{L^{\frac{2}{\delta}}(I,L^{\frac{2}{1-\delta}})} \leq \, \|\nabla u\|_{L^{\infty}(I,L^{2})}^{1-2\delta} \,\|\nabla u\|_{L^{4}(I,L^{4})}^{2\delta}.
		\end{align}
		Let us turn to the second term. For $t \in \Jbold$, we obtain using \eqref{Log} with $\beta = \frac{1}{2}-\frac{\delta}{2}$  that
		\begin{align*}
			\|{\rm e}^{\alpha|u|^2}-1\|_{L^{\infty}}^{\theta} \lesssim \left(1+\frac{\|u\|_{\mathcal{C}^{{1\over 2}-\frac{\delta}{2}}}}{\|u\|_{\mu}}\right)^{\alpha \theta \lambda \|u\|^2_\mu},
		\end{align*}
		for some $0<\mu<1$ and $\lambda>\frac{1}{\pi(1-\delta)}$ to be chosen later.
		Since $\|u\|^2_{\mu} = \|\nabla u\|^2_{L^2} + \mu^2 \|u\|^2_{L^2}< \mathcal{H}(u_0) + \mu^2 \mathcal{M}(u_0)=: K^2(\mu)$, we bound
		\[
		\|{\rm e}^{\alpha|u|^2}-1\|_{L^{\infty}}^{\theta} \lesssim \left(1+\frac{\|u\|_{\mathcal{C}^{{1\over 2}-\frac{\delta}{2}}}}{K(\mu)}\right)^{\alpha \theta \lambda K^2(\mu)}.
		\]
		Since $K^2(\mu) \rightarrow H(u_0)<\frac{2}{(1+b)(2-b)}$ as $\mu \rightarrow 0$, we can choose $0<\mu<1$ sufficiently small so that $K^2(\mu) < \frac{2}{(1+b)(2-b)}$. Moreover, as $\frac{\theta K^2(\mu)}{1-\delta} \rightarrow \frac{1+b}{2} K^2(\mu) < \frac{1}{2-b}$ as $\delta \rightarrow 0$, we choose $0<\delta<\frac{1-b}{2}$ sufficiently small such that $\frac{\theta K^2(\mu)}{1-\delta} < \frac{1}{2-b}$. At final, we choose $\frac{1}{\pi (1-\delta)}<\lambda < \frac{2}{\alpha \theta K^2(\mu)}$ so that $\alpha \theta \lambda K^2(\mu)<2$. It follows that
		\[
		\|{\rm e}^{\alpha|u|^2}-1\|_{L^{\infty}}^{\theta} \lesssim (1+ \|u\|_{\mathcal{C}^{\frac12-\frac{\delta}{2}}})^2 \lesssim \|u\|_{\mathcal{C}^{\frac12-\frac{\delta}{2}}}^{2},
		\]
		where we have used the fact that $\|u(t)\|_{\mathcal{C}^{\frac12-\frac{\delta}{2}}}\geq \|u(t)\|_{L^\infty} \geq 1$ for all $t\in \Jbold$.
		Therefore,
		\begin{align}
		\label{A2}
		\left\|\|\nabla u\|_{L^{\frac{2}{1-\delta}}} \|u\|_{L^{\frac{4}{\delta}}}^2 \|{\rm e}^{\alpha|u|^2}-1\|_{L^{\infty}}^{\theta}\right\|_{L^{\frac{2}{1+2 \delta}}(\Jbold)} &\lesssim
		\left\|\|\nabla u\|_{L^{\frac{2}{1-\delta}}} \|u\|_{L^{\frac{4}{\delta}}}^2 \|u\|_{\mathcal{C}^{\frac12-\frac{\delta}{2}}}^{2} \right\|_{L^{\frac{2}{1+2 \delta}}(I)} \nonumber\\
		&\lesssim  \|\nabla u\|_{L^{\frac{2}{\delta}}(I,L^{\frac{2}{1-\delta}})}  \|u\|_{L^{\frac{2}{\delta}}(I,L^{\frac{4}{\delta}})}^{2} \|u\|_{L^{\frac{4}{1-\delta}}(I,\mathcal{C}^{\frac12-\frac{\delta}{2}})}^{2} \nonumber\\
		&\lesssim \|u\|_{L^{\frac{2}{\delta}}(I,L^{\frac{4}{\delta}})}^{2}\, \|u\|_{S^1(I)}^3.
		\end{align}
		The last estimate follows from \eqref{interpolation} and the fact
		\begin{align*}
			\|u\|_{L^{\frac{4}{1-\delta}}(I,\mathcal{C}^{\frac12-\frac{\delta}{2}})}&\lesssim\|u\|_{L^{\frac{4}{1-\delta}}(I,W^{1, \frac{4}{1+\delta}})}\\
			&\lesssim \|u\|_{L^{\infty}(I,H^1)}^\delta\, \|u\|_{L^{4}(I,W^{1,4})}^{1-\delta}\\
			&\lesssim \|u\|_{S^1(I)}.
		\end{align*}
		Combining inequalities \eqref{A1} and \eqref{A2}, we end up with
		\begin{equation}
		\label{A}
		\|\mathcal{A}\|_{L^{\frac{2}{1+2 \delta}}(I,L^{\frac{1}{1-\delta}})} \lesssim \|u\|_{S^1(I)}  \|u\|_{L^{\frac{4}{1+\delta}}(I,L^{\frac{4}{\delta}})}^{2}+\|u\|_{L^{\frac{2}{\delta}}(I,L^{\frac{4}{\delta}})}^{2} \|u\|_{S^1(I)}^3.
		\end{equation}
		Let us now estimate the term $\|\mathcal{B}\|_{L^{\frac{2}{1+2 \delta}}(I,L^{\frac{1}{1-\delta}})}$. Taking $\frac{1}{1-\delta}<p,q<\infty$ such that $\frac1p+\frac1q=1-\delta$ and applying H\"older's inequality, we get
		\[
		\|\mathcal{B}\|_{L^{\frac{1}{1-\delta}}} \lesssim \left\||x|^{-b-1} |u| \left({\rm e}^{\alpha|u|^2} - 1\right) \right\|_{L^{\frac{1}{1-\delta}}} \lesssim \|u\|_{L^q} \left\|\frac{{\rm e}^{\alpha|u|^2} - 1}{|x|^{b+1}}\right\|_{L^p}.
		\]
		Clearly,
		\[
		\left\|\frac{{\rm e}^{\alpha|u|^2} - 1}{|x|^{b+1}}\right\|_{L^p}^p \lesssim  {\rm e}^{\alpha(p-1) \|u\|_{L^{\infty}}^2}\,\int \frac{{\rm e}^{\alpha|u|^2} - 1}{|x|^{p(b+1)}}dx.
		\]
		Since $\frac{1}{1-\delta}<\frac{2}{b+1}$ for $0<\delta<\frac{1-b}{2}$, we choose $\frac{1}{1-\delta}<p<\frac{2}{b+1}$. Hence we can apply the singular Moser-Trudinger inequality for the term $\mathlarger{\int} \frac{{\rm e}^{\alpha|u|^2} - 1}{|x|^{p(b+1)}}dx$ to obtain
		\begin{align*}
			\|\mathcal{B}\|_{L^{\frac{2}{1+2 \delta}}(I,L^{\frac{1}{1-\delta}})}  \lesssim \left\| \|u\|_{L^q} \,{\rm e}^{\alpha\frac{p-1}{p} \|u\|_{L^{\infty}}^2} \right\|_{L^{\frac{2}{1+2 \delta}}(\Ibold)}  +\left\| \|u\|_{L^q} {\rm e}^{\alpha\frac{p-1}{p} \|u\|_{L^{\infty}}^2} \right\|_{L^{\frac{2}{1+2 \delta}}(\Jbold)}.
		\end{align*}
		Note that the choice of $p$ leads to $q>\frac{2}{1-b-2\delta}$. Therefore,
		\[
		\|\mathcal{B}\|_{L^{\frac{2}{1+2 \delta}}(I,L^{\frac{1}{1-\delta}})} \lesssim \|u\|_{L^{\frac{2}{1+2 \delta}}(I,L^q)}+\|u\|_{L^{\gamma}(I,L^q)} \left\|{\rm e}^{\alpha\frac{p-1}{p} \|u\|_{L^{\infty}}^2} \right\|_{L^{\rho}(\Jbold)},
		\]
		where $\frac{2}{1+2 \delta}<\gamma,\rho<\infty$ such that $\frac{1}{\gamma}+\frac{1}{\rho}=\frac{1+2 \delta}{2}$.
		Let $t \in \Jbold$. An application of the Log estimate \eqref{Log} with $\beta =\frac{1}{2}$ gives
		\[
		{\rm e}^{\alpha\frac{p-1}{p} \|u\|_{L^{\infty}}^2} \lesssim \left( 1+ \frac{\|u\|_{\mathcal{C}^{\frac{1}{2}}}}{\|u\|_\mu} \right)^{\alpha \frac{p-1}{p}\lambda \|u\|_\mu^2},
		\]
		for some $0<\mu<1$ and $\lambda>\frac{1}{\pi}$ to be chosen shortly. Since $\|u\|_\mu^2 < K^2(\mu)$, it follows that
		\[
		{\rm e}^{\alpha\frac{p-1}{p} \|u\|_{L^{\infty}}^2} \lesssim \left( 1+ \frac{\|u\|_{\mathcal{C}^{\frac{1}{2}}}}{K(\mu)} \right)^{\alpha \frac{p-1}{p}\lambda K^2(\mu)}.
		\]		
		Choose $0<\mu<1$ sufficiently small such that $K^2(\mu)<\frac{2}{(1+b)(2-b)}<1$. Since $\frac{1}{1-\delta}<\frac{2(2-b)}{2(2-b)-1}$ for $0<\delta<\frac{1}{2(2-b)}$, we choose $\frac{1}{1-\delta}<p<\frac{2(2-b)}{2(2-b)-1}$. In particular $\frac{p}{2(2-b)(p-1)}>1>K^2(\mu)$. At final, we choose $\frac{1}{\pi}<\lambda<\frac{p}{\alpha (p-1) K^2(\mu)}$ so that $\alpha\frac{p-1}{p} \lambda K^2(\mu)<1$. We thus get
		\[
		{\rm e}^{\alpha \frac{p-1}{p} \|u\|^2_{L^\infty}} \lesssim 1+\|u\|_{\mathcal{C}^{\frac12}} \lesssim \|u\|_{W^{1,4}},
		\]
		where we have used $\|u(t)\|_{W^{1,4}} \gtrsim \|u(t)\|_{\mathcal{C}^{\frac12}} \geq \|u(t)\|_{L^\infty} \geq 1$ for all $t \in \Jbold$. Therefore, choosing $\frac{2}{1+2\delta}<\rho<4$, one gets
		$$
		\|{\rm e}^{\alpha\frac{p-1}{p} \|u\|_{L^\infty}^2}\|_{L^{\rho}(\Jbold)} \lesssim		
		\|u\|_{L^4(I,W^{1,4})}^{\frac{4}{\rho}}.
		$$
		This finally leads to
		\[
		\|\mathcal{B}\|_{L^{\frac{2}{1+2 \delta}}(I,L^{\frac{1}{1-\delta}})} \lesssim \|u\|_{L^{\frac{2}{1+2 \delta}}(I, L^q)}+\|u\|_{L^{\gamma}(I, L^q)} \|u\|_{S^1(I)}^{\frac{4}{\rho}}.
		\]
		Note also that this choice of $p$ leads to $0<\delta<\frac{1}{2(2-b)}$ and $q>\frac{2(2-b)}{1-2\delta(2-b)}$. Arguing similarly for $\| N(x,u)\|_{L^{\frac{2}{1+2 \delta}}(I,L^{\frac{1}{1-\delta}})}$, we conclude that
		\begin{align}
		\nonumber
		\|u\|_{S^1(I)} &\lesssim \|u(T)\|_{H^1}+ \|u\|_{S^1(I)}  \|u\|_{L^{\frac{4}{1+\delta}}(I,L^{\frac{4}{\delta}})}^{2}+\|u\|_{L^{\frac{2}{\delta}}(I,L^{\frac{4}{\delta}})}^{2} \|u\|_{S^1(I)}^3\\
		&\mathrel{\phantom{\lesssim \|u(T)\|_{H^1}}}+\label{uniform bound}\ \|u\|_{L^{\frac{2}{1+2 \delta}}(I, L^q)}+\|u\|_{L^{\gamma}(I,L^q)} \|u\|_{S^1(I)}^{\frac{4}{\rho}},
		\end{align}
		where $0<\delta< \min\left\{\frac{1-b}{2}, \frac{1}{2(2-b)} \right\}$, $\gamma>\frac{2}{1+2\delta}$ and $q>\max\left\{\frac{1}{1-b-2\delta}, \frac{2(2-b)}{1-2\delta(2-b)} \right\}$.
		
		Since $0<\delta< \min\left\{\frac{1-b}{2}, \frac{1}{2(2-b)} \right\}$ and $\gamma>\frac{2}{1+2\delta}$, it is easy to check that the condition \eqref{pq} is satisfied for
		\begin{align} \label{pair-pq}
		(p,q)\in \left\{ \left(\frac{4}{1+\delta}, \frac{4}{\delta}\right), \left(\frac{2}{\delta}, \frac{4}{\delta}\right), \left(\frac{2}{1+2\delta},q\right), (\gamma, q)\,\right\}
		\end{align}
		provided that $q$ satisfies an additional condition $q>\frac{4}{1-2\delta}$. We thus obtain by Corollary \ref{Bound1} that $\|u\|_{L^{p}((a,\infty),L^q)}< \infty$ for all $a>0$ and $(p,q)$ in \eqref{pair-pq}.
		
		Now let $0<\delta< \min\left\{\frac{1-b}{2}, \frac{1}{2(2-b)} \right\}$, $\gamma>\frac{2}{1+2\delta}$ and $q>\max\left\{\frac{1}{1-b-2\delta}, \frac{2(2-b)}{1-2\delta(2-b)}, \frac{4}{1-2\delta} \right\}$. We choose $T>0$ so that $\|u\|_{L^{\frac{4}{1+\delta}}((T,\infty),L^{\frac{4}{\delta}})}^2 \leq \frac12$. Set $A:=\|u\|_{L^{\frac{2}{1+2 \delta}}((T,\infty),L^q)} < \infty$. We have
		\begin{align*}
			\|u\|_{S^1((T,\infty))}& \lesssim \|u(T)\|_{H^1}+ \|u\|_{S^1((T,\infty))}  \|u\|_{L^{\frac{4}{1+\delta}}((T,\infty),L^{\frac{4}{\delta}})}^{2}+\|u\|_{L^{\frac{2}{\delta}}((T,\infty),L^{\frac{4}{\delta}})}^{2} \|u\|_{S^1((T,\infty))}^3\\
			& \mathrel{\phantom{\lesssim \|u(T)\|_{H^1}}} +\|u\|_{L^{\frac{2}{1+2 \delta}}((T,\infty),L^q)}+\|u\|_{L^{\gamma}((T,\infty),L^q)} \|u\|_{S^1((T,\infty))}^{\frac{4}{\rho}}.
		\end{align*}
		Bounding $\|u(T)\|_{H^1}$ by a constant $C(\mathcal{H}(u_0)+\mathcal{M}(u_0)))$ depending only on the mass and energy of the initial data, one infers
		\begin{align*}
			\|u\|_{S^1((T,\infty))}& \lesssim C(\mathcal{H}(u_0)+\mathcal{M}(u_0)))+ A+\|u\|_{L^{\frac{2}{\delta}}((T,\infty),L^{\frac{4}{\delta}})}^{2} \|u\|_{S^1((T,\infty))}^3\\
			& \mathrel{\phantom{\lesssim C(\mathcal{H}(u_0)+\mathcal{M}(u_0)))+ A}}  +\|u\|_{L^{\gamma}((T,\infty),L^q)} \|u\|_{S^1((T,\infty))}^{\frac{4}{\rho}}.
		\end{align*}
		Using Corollary \ref{Bound3}, one can pick $\varepsilon>0$ small (to be determined later) and a finite number of intervals $\{I_\ell\}_{\ell=1,2, \cdots, L} $, $ I_\ell \subset (T,\infty)$ such that , for all $\ell$
		$$\|u\|_{L^{\gamma}(I_\ell,L^q)}, \quad \|u\|_{L^{\frac{2}{\delta}}(I_\ell,L^{\frac{4}{\delta}})} \quad \leq \varepsilon.$$
		Thus, by \eqref{uniform bound} and since ${4\over \rho}>1$, we get
		\begin{align*}
			\|u\|_{S^1(I_\ell)} &\lesssim C(\mathcal{H}(u_0)+\mathcal{M}(u_0))+ A+\varepsilon^2 \|u\|_{S^1(I_\ell)}^3 +\varepsilon \|u\|_{S^1(I_\ell)}^{\frac{4}{\rho}}.
		\end{align*}
		A continuity argument allows us to pick $\varepsilon>0$ sufficiently small depending only on $C(\mathcal{H}(u_0)+\mathcal{M}(u_0))+ A$ such that $\|u\|_{S^1(I_\ell)} \leq C(\mathcal{H}(u_0),\mathcal{M}(u_0), A)$. Since the number of intervals is finite and the conclusion can be made for all $I_\ell$'s, we get $\|u\|_{S^1((T,\infty))}<\infty$. A similar argument applies for negative times, and we get $\|u\|_{S^1((-\infty,-S))}<\infty$ for some $S>0$. We conclude the proof by the local theory.
	\end{proof}

	\section{Global bounds 2}
	\label{S6}
	In this section, we prove the second global bound in \eqref{glo-bou}. For a time slab $I \subset \R$, we define $S^0(I)$ by
	\[
	\|u\|_{S^0(I)} = \|u\|_{L^\infty(I, L^2)} + \|u\|_{L^4(I, L^4)}.
	\]
	\begin{theo} \label{theo-glo-bou-2}
		Let $u_0 \in \Sigma$ be such that $\mathcal{H}(u_0)<\frac{2}{(1+b)(2-b)}$. Let $u$ the corresponding global solution to \eqref{NLSexp} and set $w(t):=(x+2it\nabla)u(t)$. Then it holds that $w \in S^0(\R)$.
	\end{theo}
	\begin{proof}
		Let $T>0$ and set $I=(T,\infty)$. Since $x+2it\nabla$ commutes with $i\partial_t+\Delta$, the Duhamel formula implies
		\[
		w(t)={\rm e}^{i(t-T)\Delta}w(T)-i\int_{T}^{t} {\rm e}^{i(t-s)\Delta} (x+2is\nabla)N(x,u) \, ds.
		\]
		Let $v(t,x):={\rm e}^{-i\frac{|x|^2}{4t}}u(t,x)$. It is easy to see that $|v|=|u|$, $|(x+2it\nabla)N(x,u)|=2|t| |\nabla N(x,v)|$ and $2|t| |\nabla v|=|w|$. 
		We thus bound
		\begin{align*}
		|(x+2is\nabla)N(x,u)|&
		\lesssim |x|^{-b} (2 |s||\nabla v| )|u|^2 \left({\rm e}^{\alpha|u|^2} - 1\right)+ |x|^{-b-1} (2 |s||u|^3) \left({\rm e}^{\alpha|u|^2} - 1\right)
		:= \mathcal{C}+\mathcal{D},
		\end{align*}
		where we have used the fact that for all $x\geq 0$, ${\rm e}^x-1-x\leq x({\rm e}^x-1)$. It follows from the pseudo-conformal law that $\|w(T)\|_{L^2} \leq \|xu_0\|_{L^2}$. Thus, Strichartz estimate yields
		\begin{align*}
		\|w\|_{S^0(I)} &\lesssim \|xu_0\|_{L^2}+\| (x+2is\nabla)N(x,u)\|_{L^{\frac{2}{1+2 \delta}}(I,L^{\frac{1}{1-\delta}})} \\
		&\lesssim \|xu_0\|_{L^2}+ \| 2|s| |\nabla N(x,v)| \|_{L^{\frac{2}{1+2 \delta}}(I,L^{\frac{1}{1-\delta}})}.
		\end{align*}
		As above, one gets
		\[
		\|\mathcal{C}\|_{L^{\frac{2}{1+2 \delta}}(I,L^{\frac{1}{1-\delta}})} \lesssim \|w\|_{S^0(I)} \|u\|_{L^{\frac{4}{1+\delta}}(I,L^{\frac{4}{\delta}})}^{2}+\|u\|_{L^{\frac{2}{\delta}}(I,L^{\frac{4}{\delta}})}^{2} \|w\|_{S^0(I)}\|u\|_{S^1(I)}^2
		\]
		and
		\[
		\|\mathcal{D}\|_{L^{\frac{2}{1+2 \delta}}(I,L^{\frac{1}{1-\delta}})} \lesssim \||s||u|^3\|_{L^{\frac{2}{1+2 \delta}}(I,L^q)}+\||s||u|^3\|_{L^{\gamma}(I, L^q)} \|u\|_{S^1(I)}^{\frac{4}{\rho}},
		\]
		where $0<\delta< \min\left\{\frac{1-b}{2}, \frac{1}{2(2-b)} \right\}$, $\gamma>\frac{2}{1+2\delta}$ and $q>\max\left\{\frac{1}{1-b-2\delta}, \frac{2(2-b)}{1-2\delta(2-b)} \right\}$. The last inequality can be written as
		\[
		\|\mathcal{D}\|_{L^{\frac{2}{1+2 \delta}}(I,L^{\frac{1}{1-\delta}})} \lesssim \||s|^{\frac13}|u|\|_{L^{\frac{6}{1+2 \delta}}(I,L^{3q})}^3+\||s|^{\frac13}|u|\|_{L^{3\gamma}(I,L^{3q})}^3 \|u\|_{S^1(I)}^{\frac{4}{\rho}}.
		\]
		As in Corollary $\ref{Bound1}$, we note that for all $a>0$, the norm $\||s|^{\frac{1}{3}} u\|_{L^m((a,\infty),L^n)}<\infty$ provided that $1 \leq m <\infty$, $2 \leq n <\infty$ satisfying
		\begin{align} 
		\label{mn}
		m\left(\frac{2}{3}-\frac{2}{n}\right)>1.
		\end{align}
		Since $\gamma>\frac{2}{1+2\delta}$, the condition \eqref{mn} is fulfilled for $(m,n) = \left\{ \left(\frac{6}{1+2\delta},3q\right), (3\gamma, 3q)\right\}$ provided that $q>\frac{4}{3-2\delta}$. Under the conditions 
		$$0<\delta< \min\left\{\frac{1-b}{2}, \frac{1}{2(2-b)} \right\}, \quad \gamma>\frac{2}{1+2\delta}, \quad q>\max\left\{\frac{1}{1-b-2\delta}, \frac{2(2-b)}{1-2\delta(2-b)}, \frac{4}{3-2\delta} \right\},$$ 
		we argue as above to obtain $\|w\|_{S^0((T,\infty))}<\infty$ for some $T>0$. By the same argument, we prove as well that $\|w\|_{S^0((-\infty,-S))}<\infty$ for some $S>0$. It remains to show that $w \in S^0([-S,T])$. The proof of the latter claim follows the same argument as in \cite{TVZ}. To see this, set $H(t)=x+2it\nabla$. We are going to prove that $\|Hu\|_{S^0([-S,T])}<\infty$. Divide $[-S,T]$ into a finite number of intervals $J_k=[t_k,t_{k+1}]$ such that $|J_k|\leq \varepsilon$, where $\varepsilon>0$ is to be chosen later. The Duhamel formula reads
		\[
		H(t)u(t)={\rm e}^{i(t-t_k)\Delta}H(t_k)u(t_k)-i\int_{t_k}^{t} {\rm e}^{i(t-s)\Delta} H(s)N(x,u) \, ds.
		\]
		By Strichartz estimates,
		\begin{align*}
		\|Hu\|_{S^0(J_k)} &\lesssim \|H(t_k)u(t_k)\|_{L^2}+\| H(s)N(x,u)\|_{L^{\frac{2}{1+2 \delta}}(J_k,L^{\frac{1}{1-\delta}})} \\
		&\lesssim \|H(t_k)u(t_k)\|_{L^2}+ \| 2|s| |\nabla N(x,v)| \|_{L^{\frac{2}{1+2 \delta}}(J_k,L^{\frac{1}{1-\delta}})}.
		\end{align*}
		Note that in the following all constants involved in $\lesssim$ are independent of $k$.	Using the fact that, for all $x\geq 0$ and all $\eta>0$, $x({\rm e}^x-1) \leq \frac{{\rm e}^{(1+\eta)x}-1}{\eta}$, and that $|v|=|u|$, we bound
		\[
		2 |s| |\nabla N(x,v)| \lesssim_{\eta} |x|^{-b} (2 |s||\nabla v| ) \left({\rm e}^{\alpha (1+\eta)|u|^2} - 1\right)+ |x|^{-b-1} (2 |s||u|) \left({\rm e}^{\alpha (1+\eta) |u|^2} - 1\right).
		\]
		The first term in the right hand side is estimated as follows. By H\"older's inequality, 
		\[
		\left\|2 |s||\nabla v|  \frac{{\rm e}^{\alpha (1+\eta)|u|^2}-1}{|x|^{b}} \right\|_{L^{\frac{1}{1-\delta}}} \lesssim \|2 |s||\nabla v|\|_{L^{\frac{2}{1-\delta}}} \bigg\|\frac{{\rm e}^{\alpha (1+\eta)|u|^2}-1}{|x|^{b}}\bigg\|_{L^{\frac{2}{1-\delta}}}.
		\]
		Hence
		\[
		\left\| 2 |s||\nabla v|  \frac{{\rm e}^{\alpha (1+\eta) |u|^2}-1}{|x|^{b}}\right\|_{L^{\frac{2}{1+2 \delta}}(J_k,L^{\frac{1}{1-\delta}})} \lesssim \|2 |s||\nabla v|\|_{L^{\frac{2}{\delta}}(J_k,L^{\frac{2}{1-\delta}})} \bigg\|\frac{{\rm e}^{\alpha (1+\eta)|u|^2}-1}{|x|^{b}}\bigg\|_{L^{\frac{2}{1+\delta}}(J_k,L^{\frac{2}{1-\delta}})}.
		\]
		By \eqref{interpolation}, $\|2 |s||\nabla v|\|_{L^{\frac{2}{\delta}}(J_k,L^{\frac{2}{1-\delta}})}=\|Hu\|_{L^{\frac{2}{\delta}}(J_k,L^{\frac{2}{1-\delta}})} \lesssim \|Hu\|_{S^0(J_k)}$. Write
		\[
		\bigg\|\frac{{\rm e}^{\alpha (1+\eta) |u|^2} - 1}{|x|^{b}}\bigg\|_{L^{\frac{2}{1-\delta}}}^{\frac{2}{1-\delta}} \lesssim  \left({\rm e}^{\alpha (1+\eta) \|u\|_{L^{\infty}}^2}-1\right)^{\frac{1+\delta}{1-\delta} }\int \frac{{\rm e}^{\alpha (1+\eta)|u|^2} - 1}{|x|^{\frac{2b}{1-\delta}}} dx.
		\]
		Since $\frac{2b}{1-\delta} \rightarrow 2b <2$ as $\delta \rightarrow 0$ and $\frac{1}{\mathcal{H}(u_0)}>1$, one can choose $0<\delta<\frac12$ and $\eta>0$ sufficiently small such that $\frac{2b}{1-\delta}<2$ and $0<\eta<\frac{1}{\mathcal{H}(u_0)}-1$. This guarantees that $\|\nabla (\sqrt{1+\eta} \, u)\|_{L^2} \leq \sqrt{1+\eta} \sqrt{\mathcal{H}(u_0)} <1$. Hence we can apply the singular Moser-Trudinger inequality for the term $\mathlarger{\int} \frac{{\rm e}^{\alpha (1+\eta)|u|^2} - 1}{|x|^{\frac{2b}{1-\delta}}} dx$. Thus
		\[
		\bigg\|\frac{{\rm e}^{\alpha (1+\eta)|u|^2}-1}{|x|^{b}}\bigg\|_{L^{\frac{2}{1+\delta}}(J_k,L^{\frac{2}{1-\delta}})} \lesssim \|{\rm e}^{\alpha (1+\eta) \|u\|_{L^{\infty}}^2}-1\|_{L^1(\Ibold_k)}^{\frac{1+\delta}{2}} + \|{\rm e}^{\alpha (1+\eta) \|u\|_{L^{\infty}}^2}-1\|_{L^1(\Jbold_k)}^{\frac{1+\delta}{2}},
		\]
		where $\Ibold_k := \{t \in J_k / \, \|u(t)\|_{L^{\infty} }  \leq 1 \}$ and $\Jbold_k:= \{t \in J_k / \, \|u(t)\|_{L^{\infty} }  \geq 1 \}$.	Let $t \in \Ibold_k$. We have
		\[
		{\rm e}^{\alpha (1+\eta) \|u\|_{L^{\infty}}^2}-1 \lesssim_{\alpha,\eta} 1.
		\]
		Thus
		\[
		\|{\rm e}^{\alpha (1+\eta)\|u\|_{L^{\infty}}^2}-1\|_{L^{1}(\Ibold_k)} \lesssim |J_k|.
		\]
		Let $t \in \Jbold_k$. An application of the Log estimate \eqref{Log} gives
		\[
		\|{\rm e}^{\alpha (1+\eta) |u|^2}-1\|_{L^{\infty}} \lesssim \left(1+\frac{\|u\|_{\mathcal{C}^{\frac12}}}{K(\mu)}\right)^{\alpha (1+\eta) \lambda K^2(\mu)},
		\]
		where $K^2(\mu)=\mathcal{H}(u_0)+\mu^2 \mathcal{M}(u_0)$, $0<\mu<1$ and $\lambda>\frac{1}{\pi}$. Choose $\frac{4}{b+1}<\sigma<4$. We next choose $0<\mu<1$ sufficiently small such that $K^2(\mu)<\frac{\sigma}{2(2-b)}$. This is possible since $K^2(\mu) \rightarrow \mathcal{H}(u_0) <\frac{2}{(1+b)(2-b)}<\frac{\sigma}{2(2-b)}$. Choose $\eta>0$ sufficiently small such that $1+\eta < \frac{\sigma}{2(2-b)K^2(\mu)}$. Thus $1<\frac{\sigma}{2(2-b)(1+\eta)K^2(\mu)}$. One can thus choose $\frac{1}{\pi}<\lambda<\frac{\sigma}{\alpha(1+\eta)K^2(\mu)}$ so that $\alpha (1+\eta) \lambda K^2(\mu)<\sigma$. As above, one comes to
		\[
		\left\|{\rm e}^{\alpha (1+\eta)\|u\|_{L^{\infty}}^2}-1 \right\|_{L^{1}(\Jbold_k)} \lesssim |J_k|^{1-\frac{\sigma}{4}}\|u\|_{L^4(J_k,W^{1,4})}^{\sigma}.
		\]
		Therefore,
		\[
		\bigg\|\frac{{\rm e}^{\alpha (1+\eta)|u|^2}-1}{|x|^{b}}\bigg\|_{L^{\frac{2}{1+\delta}}(J_k,L^{\frac{2}{1-\delta}})} \lesssim |J_k|^{\frac{1+\delta}{2}}+|J_k|^{\left(1-\frac{\sigma}{4}\right)\frac{1+\delta}{2}}\|u\|_{S^1(J_k)}^{\frac{(1+\delta)\sigma}{2}}.
		\]
		Conclusion
		\[
		\left\| 2 |s||\nabla v|  \frac{{\rm e}^{\alpha (1+\eta)|u|^2}-1}{|x|^{b}}\right\|_{L^{\frac{2}{1+2 \delta}}(J_k,L^{\frac{1}{1-\delta}})} \lesssim \|Hu\|_{S^0(J_k)} \left(|J_k|^{\frac{1+\delta}{2}}+|J_k|^{\left(1-\frac{\sigma}{4}\right)\frac{1+\delta}{2}}\|u\|_{S^1(J_k)}^{\frac{(1+\delta)\sigma}{2}}\right).
		\]
		For the second term, we estimate
		\[
		\left\|  |s||u|  \frac{{\rm e}^{\alpha (1+\eta) |u|^2}-1}{|x|^{b+1}}\right\|_{L^{\frac{2}{1+2 \delta}}(J_k,L^{\frac{1}{1-\delta}})} \lesssim \| |s| |u|\|_{L^{\gamma}(J_k, L^q)} \bigg\|\frac{{\rm e}^{\alpha (1+\eta) |u|^2}-1}{|x|^{b+1}}\bigg\|_{L^{\rho} (J_k,L^{p})},
		\]
		where $\frac{1}{p}+\frac{1}{q}=1-\delta$ and $\frac{1}{\gamma}+\frac{1}{\rho}=\frac{1+2\delta}{2}$. Write
		\[
		\bigg\|\frac{{\rm e}^{\alpha (1+\eta)|u|^2} - 1}{|x|^{b+1}}\bigg\|_{L^p}^p \lesssim  \left({\rm e}^{\alpha (1+\eta)\|u\|_{L^{\infty}}^2}-1\right)^{p-1} \int \frac{{\rm e}^{\alpha (1+\delta) |u|^2} - 1}{|x|^{p(b+1)}}dx.
		\]
		Since $\frac{1}{1-\delta} \rightarrow 1 <\frac{2}{b+1}$ and $\frac{1}{\mathcal{H}(u_0)}>1$, one can choose $0<\delta<\frac12$ and $\eta>0$ sufficiently small such that $\frac{1}{1-\delta}<\frac{2}{b+1}$  and $0<\eta<\frac{1}{\mathcal{H}(u_0)}-1$. This guarantees that $\|\nabla (\sqrt{1+\eta} \, u)\|_{L^2} \leq \sqrt{1+\eta} \sqrt{\mathcal{H}(u_0)} <1$. Choose $\frac{1}{1-\delta}<p<\frac{2}{b+1}$. Hence we can apply the singular Moser-Trudinger inequality for the term $\mathlarger{\int} \frac{{\rm e}^{\alpha (1+\eta) |u|^2} - 1}{|x|^{p(b+1)}}dx$.
		Thus
		\[
		\bigg\|\frac{{\rm e}^{\alpha (1+\eta) |u|^2}-1}{|x|^{b+1}}\bigg\|_{L^{\rho}(J_k, L^{p})} \lesssim \bigg\|\left({\rm e}^{\alpha (1+\eta) \|u\|_{L^{\infty}}^2}-1\right)^{\frac{p-1}{p}} \bigg\|_{L^{\rho}(\Ibold_k)} + \bigg\|\left({\rm e}^{\alpha (1+\eta) \|u\|_{L^{\infty}}^2}-1\right)^{\frac{p-1}{p}} \bigg\|_{L^{\rho}(\Jbold_k)}.
		\]
		Let $t \in \Ibold_k$. Since
		\[
		\left({\rm e}^{\alpha (1+\eta) \|u\|_{L^{\infty}}^2}-1 \right)^{\frac{p-1}{p}\rho}  \lesssim_{\alpha,\eta, p, \rho} 1,
		\]
		we get
		\[
		\bigg\|\left({\rm e}^{\alpha (1+\eta) \|u\|_{L^{\infty}}^2}-1\right)^{\frac{p-1}{p}} \bigg\|_{L^{\rho}(\Ibold_k)}^{\rho} \lesssim |J_k|.
		\]
		Let $t \in \Jbold_k$. An application of the Log estimate \eqref{Log} with $\beta=\frac12$ gives
		\[
		\|{\rm e}^{\alpha (1+\eta)|u|^2}-1\|_{L^{\infty}}^{\frac{p-1}{p}\rho} \lesssim \left(1+\frac{\|u\|_{\mathcal{C}^{\frac12}}}{K(\mu)}\right)^{\alpha (1+\eta)  \frac{p-1}{p}\rho \lambda K^2(\mu)},
		\]
		for some $0<\mu<1$ and $\lambda>\frac{1}{\pi}$. Choose $0<\mu<1$ sufficiently small such that $K^2(\mu)<1$. Since $\frac{1}{1-\delta} \rightarrow 1 <\frac{2(2-b)}{2(2-b)-1}$ as $\delta\rightarrow 0$, one can choose $0<\delta<\frac12$ sufficiently small such that $\frac{1}{1-\delta}<\frac{2(2-b)}{2(2-b)-1}$. Choose $\frac{1}{1-\delta}<p<\frac{2(2-b)}{2(2-b)-1}$. In particular $\frac{p}{2(2-b)(p-1)}>1>K^2(\mu)$. Choose $0<\eta<1$ such that $1+\eta<\frac{p}{2(2-b)(p-1)K^2(\mu)}$. At final, we choose $\frac{1}{\pi}<\lambda<\frac{p}{\alpha (1+\eta) (p-1) K^2(\mu)}$. Therefore, choosing $\frac{2}{1+2\delta}<\rho<4$, one gets
		\[
		\|{\rm e}^{\alpha (1+\eta) |u|^2}-1\|_{L^{\infty}}^{\frac{p-1}{p}\rho} \lesssim \left(1+\|u\|_{\mathcal{C}^{\frac12}}\right)^{4}.
		\]
		Using the fact that $1 \leq \|u(t)\|_{L^{\infty}} \leq \|u(t)\|_{\mathcal{C}^{\frac12}} \leq \|u(t)\|_{W^{1,4}}$ for all $t\in \Jbold_k$, one gets
		\[
		\|{\rm e}^{\alpha (1+\eta)|u|^2}-1\|_{L^{\infty}}^{\frac{p-1}{p}\rho} \lesssim \|u\|_{W^{1,4}}^{4}.
		\]
		We come to
		\[
		\bigg\|\left({\rm e}^{\alpha (1+\eta) \|u\|_{L^{\infty}}^2}-1\right)^{\frac{p-1}{p}} \bigg\|_{L^{\rho}(\Jbold_k)}^{\rho} \lesssim \|u\|_{L^4(J_k,W^{1,4})}^{4}.
		\]
		Thus
		\[
		\bigg\|\frac{{\rm e}^{\alpha (1+\eta) |u|^2}-1}{|x|^{b+1}}\bigg\|_{L^{\rho}(J_k, L^{p})} \lesssim \left(|J_k|+\|u\|_{S^1(J_k)}^{4}\right)^{\frac{1}{\rho}}.
		\]
		By the sobolev embedding, one has
		\[ \| |s| |u|\|_{L^{\gamma}(J_k, L^q)} \lesssim \|u\|_{L^{\infty}(J_k, H^1)} \|s\|_{L^{\gamma}(J_k)} \lesssim |J_k|^{\frac{\gamma+1}{\gamma}},
		\]
		where we have used the conservation laws and the fact that $\|s\|_{L^{\gamma}(J_k)}=\left(\frac{t_{k+1}^{\gamma+1}-t_k^{\gamma+1}}{\gamma+1}\right)^{\frac{1}{\gamma}} \lesssim (t_{k+1}-t_{k})^{\frac{\gamma+1}{\gamma}}=|J_k|^{\frac{\gamma+1}{\gamma}}$. Therefore,
		\[
		\left\|  |s||u|  \frac{{\rm e}^{\alpha (1+\eta) |u|^2}-1}{|x|^{b+1}} \right\|_{L^{\frac{2}{1+2 \delta}}(J_k,L^{\frac{1}{1-\delta}})} \lesssim |J_k|^{\frac{\gamma+1}{\gamma}} \left(|J_k|+\|u\|_{S^1(J_k)}^{4}\right)^{\frac{1}{\rho}}.
		\]
		Collecting the above estimates, we obtain
		\begin{align*}
		\|Hu\|_{S^0(J_k)} \lesssim \|H(t_k)u(t_k)\|_{L^2} &+ \|Hu\|_{S^0(J_k)} \left(|J_k|^{\frac{1+\delta}{2}}+|J_k|^{\left(1-\frac{\sigma}{4}\right)\frac{1+\delta}{2}}\|u\|_{S^1(J_k)}^{\frac{(1+\delta)\sigma}{2}}\right) \\
		&+|J_k|^{\frac{\gamma+1}{\gamma}} \left(|J_k|+\|u\|_{S^1(J_k)}^{4}\right)^{\frac{1}{\rho}} \\
		\lesssim \|H(t_k) u(t_k)\|_{L^2} &+ \|Hu\|_{S^0(J_k)} \left( \varepsilon^{\frac{1+\delta}{2}} + \varepsilon^{\left(1-\frac{\sigma}{4}\right)\frac{1+\delta}{2}} \|u\|^{\frac{(1+\delta)\sigma}{2}}_{S^1(J_k)} \right) \\
		& + \varepsilon^{\frac{\gamma+1}{\gamma}} \left(\varepsilon + \|u\|^4_{S^1(J_k)}\right)^{\frac{1}{\rho}}.
		\end{align*}
		Since $\|u\|_{S^1(\R)}<\infty$, we can choose $\varepsilon>0$ small enough depending on $S,T$ and $\|u\|_{S^1(\R)}$ to get
		\[
		\|Hu\|_{S^0(J_k)} \leq C\|H(t_k) u(t_k)\|_{L^2} + C,
		\]
		for some constant $C>0$ independent of $S$ and $T$. By induction, we obtain for each $k$,
		\[
		\|Hu\|_{S^0(J_k)} \leq C\|H(-S)u(-S)\|_{L^2} + C.
		\]
		Summing over all subintervals $J_k$, we prove $\|Hu\|_{S^0([-S,T])} <\infty$. The proof is complete.
	\end{proof}


	\section{Scattering in weighted $L^2$ space}
	\label{S7}
	In this section, we give the proof of our main result in Theorem $\ref{Main-NLS}$. 
	
	\noindent \textit{Proof of Theorem $\ref{Main-NLS}$.} Let $u_0 \in \Sigma$ and $u$ the corresponding global solution to \eqref{NLSexp}. By Duhamel formula, we have
	\[
	{\rm e}^{-it\Delta} u(t) = u_0 -i \int_0^t {\rm e}^{-is\Delta} N(x,u) ds.
	\]
	Let $0<t_1<t_2<+\infty$. It follows from Strichartz estimates that
	\begin{align*}
	\|{\rm e}^{-it_2 \Delta} u(t_2) - {\rm e}^{-it_1\Delta} u(t_1)\|_{H^1} &= \left\| \int_{t_1}^{t_2} {\rm e}^{-is\Delta} N(x,u) ds \right\|_{H^1} \\
	&\lesssim \|N(x,u)\|_{L^{\frac{2}{1+2\delta}}((t_1,t_2), L^{\frac{1}{1-\delta}})}+\|\nabla N(x,u)\|_{L^{\frac{2}{1+2\delta}}((t_1,t_2), L^{\frac{1}{1-\delta}})}.
	\end{align*}
	Arguing as in the proof of \eqref{uniform bound}, we obtain
	\begin{align} \label{scattering}
	\|{\rm e}^{-it_2 \Delta} u(t_2) - {\rm e}^{-it_1 \Delta} u(t_1)\|_{H^1} &\lesssim \|u\|_{S^1((t_1,t_2))}  \|u\|_{L^{\frac{4}{1+\delta}}((t_1,t_2),L^{\frac{4}{\delta}})}^{2}+\|u\|_{L^{\frac{2}{\delta}}((t_1,t_2),L^{\frac{4}{\delta}})}^{2} \|u\|_{S^1((t_1,t_2))}^3 \nonumber \\
	&\mathrel{\phantom{\lesssim}}+ \|u\|_{L^{\frac{2}{1+2 \delta}}((t_1,t_2), L^q)}+\|u\|_{L^{\gamma}((t_1,t_2),L^q)} \|u\|_{S^1((t_1,t_2))}^{\frac{4}{\rho}},
	\end{align}
	where $0<\delta< \min\left\{\frac{1-b}{2}, \frac{1}{2(2-b)} \right\}$, $\gamma>\frac{2}{1+2\delta}$ and $q>\max\left\{\frac{1}{1-b-2\delta}, \frac{2(2-b)}{1-2\delta(2-b)} \right\}$. By adding an additional condition $q>\frac{4}{1-2\delta}$, we learn from Corollary $\ref{Bound1}$ that $\|u\|_{L^{p}((a,\infty),L^q)}< \infty$ for all $a>0$ and $(p,q)$ in \eqref{pair-pq}. In particular, $\|u\|_{L^{\frac{2}{1+2\delta}}((t_1,t_2), L^q)} \rightarrow 0$ as $t_1 \rightarrow +\infty$. Since $\|u\|_{S^1(\R)} <\infty$, we infer that the right hand side of \eqref{scattering} tends to zero as $t_1,t_2 \rightarrow +\infty$ provided that $0<\delta< \min\left\{\frac{1-b}{2}, \frac{1}{2(2-b)} \right\}$, $\gamma>\frac{2}{1+2\delta}$ and $q>\max\left\{\frac{1}{1-b-2\delta}, \frac{2(2-b)}{1-2\delta(2-b)}, \frac{4}{1-2\delta} \right\}$. This shows that ${\rm e}^{-it\Delta} u(t)$ is a Cauchy sequence in $H^1$ as $t\rightarrow +\infty$. There thus exists $u_0^+ \in H^1$ such that ${\rm e}^{-it\Delta} u(t) \rightarrow u_0^+$ as $t\rightarrow +\infty$. It remains to show that this scattering state $u_0^+$ belongs to $\Sigma$. Since $x+2 i t \nabla$ commutes with $i \partial_t +\Delta u$, the Duhamel formula gives
	\[
	(x+ 2i t \nabla )u(t) = {\rm e}^{it\Delta} x u_0 - i \int_0^t {\rm e}^{i(t-s)\Delta} (x+ 2is\nabla) N(x,u) ds.
	\]
	Using the fact that $x+2it\nabla = {\rm e}^{it\Delta} x {\rm e}^{-it\Delta}$, we write
	\[
	x {\rm e}^{-it\Delta} u(t) = x u_0 - i \int_0^t {\rm e}^{-is\Delta} (x+2is\nabla) N(x,u) ds. 
	\]
	By Strichartz estimates, we have
	\begin{align*}
	\|x {\rm e}^{-it_2\Delta} u(t_2) - x {\rm e}^{-it_1\Delta} u(t_1)\|_{L^2} &= \left\| \int_{t_1}^{t_2} {\rm e}^{-is \Delta} (x +2is\nabla) N(x,u) ds \right\|_{L^2} \\
	&\lesssim \|(x+2is\nabla) N(x,u)\|_{L^{\frac{2}{1+2\delta}}((t_1,t_2), L^{\frac{1}{1-\delta}})} \\
	&\lesssim \|2|s||\nabla N(x,v)|\|_{L^{\frac{2}{1+2\delta}}((t_1,t_2), L^{\frac{1}{1-\delta}})},
	\end{align*}
	where $v(t,x):={\rm e}^{-i\frac{|x|^2}{4t}}u(t,x)$. Estimating as in the proof of Theorem $\ref{theo-glo-bou-2}$, we get
	\begin{align} \label{scattering-sigma}
	\|x {\rm e}^{-it_2\Delta} u(t_2) - x {\rm e}^{-it_1\Delta} u(t_1)\|_{L^2} &\lesssim \|w\|_{S^0(I)} \|u\|^2_{L^{\frac{4}{1+\delta}}(I, L^{\frac{4}{\delta}})} + \|u\|^2_{L^{\frac{2}{\delta}}(I, L^{\frac{4}{\delta}})} \|w\|_{S^0(I)} \|u\|^2_{S^1(I)} \nonumber \\
	&\mathrel{\phantom{\lesssim}} + \||s|^{\frac{1}{3}} |u|\|^3_{L^{\frac{6}{1+2\delta}}(I, L^{3q})} + \||s|^{\frac{1}{3}} |u|\|^3_{L^{3\gamma}(I, L^{3q})} \|u\|^{\frac{4}{\rho}}_{S^1(I)},
	\end{align}
	where $w(t)= (x+2it\nabla) u(t)$, $I=(t_1,t_2)$, $0<\delta < \min\left\{\frac{1-b}{2}, \frac{1}{2(2-b)}\right\}$, $\gamma>\frac{2}{1+2\delta}$ and $q>\max\left\{ \frac{1}{1-b-2\delta}, \frac{2(2-b)}{1-2\delta(2-b)}\right\}$. Since $0<\delta<\min\left\{\frac{1-b}{2}, \frac{1}{2(2-b)}\right\}$, Corollary $\ref{Bound1}$ implies that the norms $\|u\|_{L^{\frac{4}{1+\delta}}((a,\infty), L^{\frac{4}{\delta}})}$ and $\|u\|_{L^{\frac{2}{\delta}}((a,\infty), L^{\frac{4}{\delta}})}$ are finite for any $a>0$. Moreover, adding an additional condition $q>\frac{4}{3-2\delta}$, the condition \eqref{mn} is satisfied for $(m,n) = \left\{ \left(\frac{6}{1+2\delta}, 3q\right), (3\gamma, 3q)\right\}$. Thus the norms $\||s|^{\frac{1}{3}} |u|\|_{L^{\frac{6}{1+2\delta}}((a,\infty), L^{3q})}$ and $\||s|^{\frac{1}{3}} |u|\|_{L^{3\gamma}((a,\infty),L^{3q})}$ are both finite for any $a>0$. In particular, $\||s|^{\frac{1}{3}} |u|\|_{L^{\frac{6}{1+2\delta}}((t_1,t_2), L^{3q})} \rightarrow 0$ as $t_1 \rightarrow +\infty$. Since $\|u\|_{S^1(\R)} <\infty$ and $\|w\|_{S^0(\R)} <\infty$, the right hand side of \eqref{scattering-sigma} tends to zero as $t_1,t_2 \rightarrow +\infty$. This implies that $x {\rm e}^{-it\Delta} u(t)$ is a Cauchy sequence in $L^2$ as $t\rightarrow +\infty$. We thus have $x u_0^+ \in L^2$ and so $u_0^+ \in \Sigma$. Moreover, we have
	\[
	u_0^+(t) = u_0 - i \int_t^{+\infty} {\rm e}^{-is\Delta} N(x,u) ds.
	\]
	Repeating the above argument, we prove that
	\[
	\|{\rm e}^{-it\Delta} u(t) - u_0^+\|_{\Sigma} \rightarrow 0 \text{ as } t \rightarrow +\infty.
	\]
	This completes the proof for positive times, the one for negative times is similar. 
	\hfill $\Box$
	
	\appendix
	\renewcommand*{\thesection}{\Alph{section}}
	\section{Lorentz spaces}
	\label{SA}
	We recall some basic facts about the Lorentz spaces which are relevant to our
	study. We refer the reader to \cite{Berg, Gra, Lem, Hunt, Stein} and references therein for more properties and
	information on Lorentz spaces.
	\begin{defi}
		Let $u: \R^N\rightarrow \mathbb{R}$ be a measurable function. The distribution function of $u$ is given by
		$$
		{\mathbf d}_{u}(\lambda):=|\{\, x\in\R^N;\;\; |u(x)|>\lambda\,\}|, \qquad \lambda \in (0,\infty).
		$$
		Here, the notation $|E|$ stands for the $N$-dimensional Lebesgue measure of $E$. The (unidimensional) decreasing rearrangement of $u$, denoted by $u^{*}$, is defined by
		$$u^{*}(s) =\inf\,\left\{\, \lambda>0;\;\;\; \mathbf{d}_{u}(\lambda)<s\,\right\},\qquad s>0.$$
	\end{defi}
	It is clear that $\mathbf{d}_u$ and $u^{*}$ are non-negative non-increasing functions. The Lorentz spaces $L^{p,q}(\R^N)$ are defined as follows.
	\begin{defi}
		\label{Lorent}
		Let $0<p<\infty$ and $0<q\leq\infty$. Then
		$$
		L^{p,q}(\R^N)=\left\{\, u \;\;\;\mbox{measurable};\;\;\; \|u\|_{L^{p,q}}<\infty\,\right\},
		$$
		where
		\begin{eqnarray*}
			\|u\|_{L^{p,q}}&=&\; \left\{
			\begin{array}{cllll} \left(\frac{q}{p}\int_0^\infty\,\left(s^{1/p} u^{*}(s)\right)^{q}\frac{ds}{s}\right)^{1/q}\quad&\mbox{if}&\quad
				0<p,q<\infty\\\\ \displaystyle\sup_{s>0}\,\left(s^{1/p}u^{*}(s)\right) \quad
				&\mbox{if}&\quad  0<p<\infty,\;\; q=\infty
			\end{array}
			\right.
		\end{eqnarray*}
	\end{defi}
	We have $L^{p,p}=L^p$ and by convention $L^{\infty, \infty}=L^\infty$. Another way to define the Lorentz space $L^{p, q}$ is via real interpolation theory as follows (see \cite{BenSha})
	$$
	L^{p,q}=\Big[L^1, L^\infty\Big]_{1-{1\over p}, q},\quad 1<p<\infty,\;\; 1\leq q\leq\infty.
	$$
	One of the difficulties in our problem is the singular weight $|x|^{-b}$ in the nonlinearity. Since this weight does not belong to any Lebesgue space we have to treat it differently. Fortunately, $|x|^{-b}$ belongs to the Lorentz space $L^{{2\over b}, \infty}(\R^2)$ which plays an important role in our proof.
	
	The following lemma will be useful.
	\begin{lem}
		\label{IneqLor}
		Let $1<p<\infty$, $1<p_1<\infty$ and $1\leq p_2\leq\infty$ be such that
		$$
		\frac{1}{p}=\frac{1}{p_1}+\frac{1}{p_2}\,.
		$$
		Then
		\begin{equation}
		\label{Lor}
		\|fg\|_{L^p}\leq\,C\|f\|_{L^1}^{1-\theta}\,\|f\|_{L^\infty}^\theta\, \|g\|_{L^{p_2,\infty}}\,
		\end{equation}
		where $\theta=1-\frac{1}{p_1}$.
	\end{lem}

	
	\section*{Acknowledgments}
	V. D. D. would like to express his deep thanks to his wife-Uyen Cong for her encouragement and support. M. M.  would like to thank K. Nakanishi for enlightening discussions concerning the interaction Morawetz inequalities. The authors would like to thank the reviewer for his/her helpful comments and suggestions.

\end{document}